\documentclass[reqno, 11pt]{amsart}
\usepackage{amsmath,mathtools}
 \usepackage{amssymb}
\usepackage{amsthm}
\usepackage{amsmath}
\usepackage{times}
\usepackage{latexsym}
\usepackage[mathscr]{eucal}

\numberwithin{equation}{section}
 
  \newtheorem{theorem}{Theorem}[section]
  \newtheorem{proposition}[theorem]{Proposition}
  \newtheorem{lemma}[theorem]{Lemma}
  \newtheorem{corollary}[theorem]{Corollary}

  \newtheorem{example}[theorem]{Example}

\title[On total mean curvatures of foliated half-lightlike submanifolds]{On total mean curvatures of foliated half-lightlike submanifolds in semi-Riemannian manifolds}

\author[Fortun\'{e} Massamba, Samuel Ssekajja ]{Fortun\'{e} Massamba*,   Samuel Ssekajja**}

\newcommand{\acr}{\newline\indent}

\address{\llap{*\,} School of Mathematics, Statistics and Computer Science\acr
 University of KwaZulu-Natal\acr
 Private Bag X01, Scottsville 3209\acr
South Africa}
\email{massfort@yahoo.fr, Massamba@ukzn.ac.za} 
\thanks{}

\address{\llap{**\,} School of Mathematics, Statistics and Computer Science\acr
 University of KwaZulu-Natal\acr
 Private Bag X01, Scottsville 3209\acr
South Africa}
\email{ssekajja.samuel.buwaga@aims-senegal.org} 
\thanks{}



\begin{document}
  
\begin{abstract}
We derive  total mean curvature integration formulae of a three co-dimensional foliation $\mathcal{F}^{n}$ on a screen integrable half-lightlike submanifold, $M^{n+1}$ in a semi-Riemannian manifold $\overline{M}^{n+3}$. We give generalized differential equations relating to mean curvatures of a totally umbilical half-lightlike submanifold admitting a totally umbilical screen distribution, and show that they are generalizations of those given by \cite{ds2}.
\end{abstract}

\maketitle

\section{Introduction} 
The rapidly growing importance of lightlike submanifolds in semi-Riemannian geometry, particularly Lorentzian geometry, and their applications to mathematical physics--like in general relativity and electromagnetism motivated the study of lightlike geometry in semi-Riemannian manifolds. More precisely, lightlike submanifolds have been shown to represent different black hole horizons (see \cite{db} and \cite{ds2} for details). Among other motivations for investing in lightlike geometry by many physicists is the idea that the universe we are living in can be viewed as a 4-dimensional hypersurface embedded in $(4+m)$-dimensional spacetime manifold, where $m$ is any arbitrary integer. There are significant differences between lightlike geometry and Riemannian geometry as shown in \cite{db} and \cite{ds2}, and many more references therein. Some of the pioneering work on this topic is due to Duggal-Bejancu \cite{db}, Duggal-Sahin \cite{ds2} and Kupeli \cite{kup}. It is upon those books that many other researchers, including but not limited to \cite{ca}, \cite{dusa1}, \cite{gup1}, \cite{Jin}, \cite{ma1}, \cite{ma2}, \cite{ms}, \cite{cohen}, have extended their theories. 

Lightlike geometry rests on a number of operators, like shape and algebraic invariants derived from them, such as trace, determinants, and in general the $r$-th mean curvature $S_{r}$. There is a great deal of work so far on the case $r=1$ (see some in \cite{db}, \cite{ds2} and many more) and as far as we know, very little has been done for the case $r>1$. This is partly due to the non-linearity of $S_{r}$  for $r>1$, and hence very complicated to study. A great deal of research on higher order mean curvatures $S_{r}$ in Riemannian geometry has been done with numerous applications, for instance see \cite{krz} and \cite{woj}. This gap has motivated our introduction of lightlike geometry of $S_{r}$ for $r>1$.  In this paper we have considered a half-lightlike submanifold admitting an integrable screen distribution, of a semi-Riemannian manifold. On it we have focused on a  codimension 3 foliation of its screen distribution and thus derived integral formulae of its total mean curvatures (see Theorems \ref{theorem3} and \ref{theorem5}). Furthermore, we have considered totally umbilical half-lightlike submanifolds, with a totally umbilical screen distribution and generalized Theorem 4.3.7 of \cite{ds2} (see Theorem \ref{thmy} and its Corollaries). The paper is organized as follows; In section \ref{prel} we summarize the basic notions on lightlike geometry necessary for other sections. In section \ref{New1} we give some basic information on Newton transformations of a foliation $\mathcal{F}$ of the screen distribution. Section \ref{dist} focuses on integration formulae of $\mathcal{F}$ and their consequences. In section \ref{screen} we discus screen umbilical half-lightlike submanifolds and generalizations of some well-known results of \cite{ds2}.

\section{Preliminaries}\label{prel}

Let $(M^{n+1},g)$ be a two-co-dimensional submanifold of a semi-Riemannian manifold $(\overline{M}^{n+3},\overline{g})$, where $g=\overline{g}|_{TM}$. The submanifold $(M^{n+1},g)$ is called a \textit{half-lightlike} if the radical distribution $\mathrm{Rad}\,T M = T M \cap TM^{\perp}$ is a vector subbundle of the tangent bundle $TM$ and the normal bundle $TM^{\perp}$ of $M$ , with rank one. Let $S(T M)$ be a \textit{screen distribution} which is a semi-Riemannian complementary distribution of $\mathrm{Rad}\,T M$ in $T M$, and also choose a \textit{screen transversal bundle} $S(TM^\perp)$, which is semi-Riemannian and complementary to $\mathrm{Rad}\, TM$ in $TM^\perp$. Then, 
 \begin{equation}\label{N5}
  T M = \mathrm{Rad}\,T M \perp S(T M),\;\;T M^{\perp} = \mathrm{Rad}\,T M \perp S(TM^{\perp}).
 \end{equation} 
We will denote by $\Gamma(\Xi)$ the set of smooth sections of the vector bundle $\Xi$.
It is well-known from \cite{db} and \cite{ds2} that for any null section $E$ of $\mathrm{Rad}\,TM$, there exists a unique null section $N$ of the orthogonal complement of $S(T M^\perp)$ in $S(T M )^\perp$  such that $g(E,N) = 1$, it follows that there exists a lightlike \textit{transversal vector bundle} $l\mathrm{tr}(TM)$ locally spanned by $N$. Let $W\in\Gamma(S(TM^{\perp}))$ be a unit vector field, then $\overline{g}(N,N) = \overline{g}(N, Z)=\overline{g}(N, W) = 0$, for any $Z\in \Gamma(S(TM))$.
 
Let $\mathrm{tr}(TM)$ be complementary (but not orthogonal) vector bundle to $TM$ in $T\overline{M}$. Then we have the following decompositions of $\mathrm{tr}(TM)$ and $T\overline{M}$
\begin{align}
 \mathrm{tr}(TM)& =l\mathrm{tr}(TM)\perp S(TM^\perp),\label{N8}\\
  T\overline{M} & =  S(TM)\perp S(TM^\perp)\perp\{\mathrm{Rad}\, TM\oplus l\mathrm{tr}(TM)\}\label{N9} .
\end{align}
It is important to note that the distribution $S(TM)$ is not unique, and is canonically isomorphic to the factor vector bundle $TM/ \mathrm{Rad}\, TM$  \cite{db}. Let $P$ be the projection of $TM$ on to $S(TM)$. Then the local Gauss-Weingarten equations of $M$ are the following;
\begin{align}
 &\overline{\nabla}_{X}Y=\nabla_{X}Y+B(X,Y)N+D(X,Y)W,\label{P1}\\
 &\overline{\nabla}_{X}N=-A_{N}X+\tau(X)N+\rho(X)W,\label{P2}\\
 &\overline{\nabla}_{X}W=-A_{W}X+\phi(X)N,\label{P3}\\
  &\nabla_{X}PY = \nabla^{*}_{X}PY + C(X,PY)E,\label{P4}\\
  &\nabla_{X}E =-A^{ * }_{E}X -\tau(X) E,\label{P5}
 \end{align}
 for all $E\in\Gamma(\mathrm{Rad}\,T M)$, $N\in\Gamma(l\mathrm{tr}(T M))$ and $W\in\Gamma(S(TM^{\perp}))$, where $\nabla$ and $\nabla^{*}$ are induced linear connections on $TM$ and $S(TM)$, respectively, $B$ and $D$ are called the local second fundamental forms of $M$, $C$  is the local second fundamental form on $S(T M)$. Furthermore,  $\{A_{N}, A_{W}\}$ and $A^{*}_{E}$ are the shape operators on $TM$ and $S(TM)$ respectively, and $\tau$, $\rho$, $\phi$ and $\delta$ are differential 1-forms on $TM$. Notice that $\nabla^{*}$ is a metric connection on $S(TM)$ while $\nabla$ is generally not a metric connection. In fact, $\nabla$ satisfies the following relation
 \begin{align}\label{P6}
  (\nabla_{X}g)(Y,Z)= B(X,Y)\lambda(Z) + B(X,Z)\lambda(Y),
 \end{align} 
 for all $X,Y,Z\in\Gamma(TM)$, where $\lambda$ is a 1-form on $TM$ given  $\lambda(\cdot)= \overline{g}(\cdot, N)$.  
It is well-known from \cite{db} and \cite{ds2} that $B$ and $D$ are independent of the choice of $S(TM)$ and they satisfy
\begin{equation}\label{P7}
 B(X, E) = 0,\;\;\;D(X, E) = −\phi(X),\;\;\forall\,  X \in\Gamma(TM).
\end{equation}
The local second fundamental forms $B$, $D$ and $C$ are related to their shape operators by the following equations
 \begin{align}
  &g(A^{*}_{E}X,Y)=B(X,Y),\;\;\;\;\; \overline{g}(A^{*}_{E}X,N)=0,\label{P8}\\
  &g(A_{W}X,Y) = \varepsilon D(X,Y)+ \phi(X)\lambda(Y),\label{P9}\\
  &g(A_{N}X,PY) = C(X,PY ),\;\;\overline{g}(A_{N} X,N) = 0,\label{P10}\\
  &\overline{g}(A_{W} X,N) = \varepsilon \rho(X),\;\;\mbox{where}\;\; \varepsilon=\overline{g}(W,W),\label{P11}
  \end{align}
 for all $X,Y\in\Gamma(TM)$.
 From equations (\ref{P8}) we deduce that $A^{*}_{E}$ is $S(TM)$-valued, self-adjoint and satisfies $A^{*}_{E}E=0$.
Let $\overline{R}$ denote the curvature tensor of $\overline{M}$, then 
\begin{align}\label{P60}
 \overline{g}(\overline{R}(X,Y)PZ,N)&=g((\nabla_{X}A_{N})Y, PZ)-g((\nabla_{Y}A_{N})X, PZ)\nonumber\\
 &+\tau(Y)C(X,PZ)-\varepsilon\tau(X)C(Y,PZ)\{\rho(Y)D(X,PZ)\nonumber\\
 &-\rho(X)D(Y,PZ)\},\;\;\;\forall\, X,Y,Z\in\Gamma(TM).
\end{align}
A half-lightlike submanifold $(M,g)$ of a semi-Riemannian manifold $\overline{M}$ is said to be totally umbilical \cite{ds2} if on each coordinate neighborhood $\mathcal{U}$ there exist smooth functions $\mathcal{H}_{1}$ and $\mathcal{H}_{2}$ on $l\mathrm{tr}(TM)$ and $S(TM^{\perp})$ respect such that 
\begin{align}\label{P61}
 B(X,Y)=\mathcal{H}_{1}g(X,Y),\;\;\;D(X,Y)=\mathcal{H}_{2}g(X,Y),\;\;\forall\, X,Y\in\Gamma(TM).
\end{align}
Furthermore, when $M$ is totally umbilical then the following relations follows by straightforward calculations
\begin{align}\label{P62}
 A_{E}^{*}X=\mathcal{H}_{1}PX,\;\;P(A_{W}X)=\varepsilon \mathcal{H}_{2}PX,\;\;D(X,E)=0,\;\;\rho(E)=0,
\end{align}
for all $X,Y\in\Gamma(TM)$.
 
Next, we suppose that $M$ is a half-lightlike submanifold of $\overline{M}$, with an integrable screen distribution $S(TM)$. Let $M'$ be a leaf of $S(TM)$. Notice that for any screen integrable half-lightlike $M$, the leaf $M'$ of $S(TM)$ is a co-dimension 3 submanifold of $\overline{M}$ whose normal bundle  is $\{\mathrm{Rad}\,TM\oplus l\mathrm{tr}(TM)\}\perp S(TM^{\perp})$. Now, using (\ref{P1}) and (\ref{P4}) we have 
\begin{align}\label{P12}
 \overline{\nabla}_{X}Y=\nabla^{*}_{X}Y + C(X,PY)E+B(X,Y)N+D(X,Y)W,
\end{align}
for all $X,Y\in\Gamma(TM')$. Since $S(TM)$ is integrable, then its leave is semi-Riemannian and hence we have 
\begin{align}\label{P13}
 \overline{\nabla}_{X}Y=\nabla^{*'}_{X}Y+h'(X,Y), \;\;\; \forall\, X,Y\in\Gamma(TM'),
\end{align}
where $h'$ and $\nabla^{*'}$ are second fundamental form and the Levi-Civita connection of $M'$ in $\overline{M}$. From (\ref{P12}) and (\ref{P13}) we can see that 
\begin{align}\label{P14}
 h'(X,Y)=C(X,PY)E+B(X,Y)N+D(X,Y)W, 
\end{align}
for all $X,Y\in\Gamma(TM')$. Since $S(TM)$ is integrable, then it is well-known from \cite{ds2} that $C$ is symmetric on $S(TM)$ and also $A_{N}$ is self-adjoint on $S(TM)$ (see Theorem 4.1.2 for details). Thus, $h'$ given by (\ref{P14}) is symmetric on $TM'$.

Let $L\in \Gamma(\{\mathrm{Rad}\,TM\oplus l\mathrm{tr}(TM)\}\perp S(TM^{\perp}))$, then we can decompose $L$ as 
\begin{align}\label{P15}
 L=aE+bN+cW,
\end{align}
for non-vanishing smooth functions on $\overline{M}$ given by $a=\overline{g}(L,N)$, $b=\overline{g}(L,E)$ and $c=\varepsilon \overline{g}(L,W)$. Suppose that $\overline{g}(L,L)>0$, then using (\ref{P15}) we obtain a unit normal vector $\widehat{W}$ to $M'$ given by 
\begin{align}\label{P16}
 \widehat{W}=\frac{1}{\overline{g}(L,L)}(aE+bN+cW)=\frac{1}{\overline{g}(L,L)}L.
\end{align}
Next we define a (1,1) tensor $\mathcal{A}_{\widehat{W}}$ in terms of the operators $A^{*}_{E}$, $A_{N}$ and $A_{W}$ by 
\begin{align}\label{P17}
 \mathcal{A}_{\widehat{W}}X=\frac{1}{\overline{g}(L,L)}(aA^{*}_{E}X+bA_{N}X+cA_{W}X),
\end{align}
for all $X\in\Gamma(TM)$. Notice that $\mathcal{A}_{\widehat{W}}$ is self-adjoint on $S(TM)$. Applying $\overline{\nabla}_{X}$ to $\widehat{W}$ and using equations (\ref{P17}) (\ref{P1}) and (\ref{P8})-(\ref{P10}), we have 
\begin{align}\label{P18}
 g(\mathcal{A}_{\widehat{W}}X, PY)=-\overline{g}(\overline{\nabla}_{X}\widehat{W}, PY), \;\;\forall\, X,Y\in\Gamma(TM).
\end{align}
Let $\nabla^{*\perp}$ be the connection on the normal bundle $\{\mathrm{Rad}\,TM\oplus l\mathrm{tr}(TM)\}\perp S(TM^{\perp})$. Then from (\ref{P18}) we have 
\begin{align}\label{P19}
 \overline{\nabla}_{X}\widehat{W}=-\mathcal{A}_{\widehat{W}}X+\nabla^{*\perp}_{X}\widehat{W}, \;\;\forall\, X\in\Gamma(TM),
\end{align}
where 
\begin{align} 
\nabla^{*\perp}_{X}\widehat{W} & = -\frac{1}{\overline{g}(L,L)}X(\overline{g}(L,L))\widehat{W} + \frac{1}{\overline{g}(L,L)}\left[\{X(a) - a\tau(X)\}E \right.\nonumber\\
&\left. + \{X(b)+b\tau(X) + c\phi(X)\}N + \{X(c) + aD(X,E) + b\rho(X)\}W\right].\nonumber
 \end{align} 
\begin{example}
 {\rm
 Let $\overline{M}=(\mathbb{R}^{5}_{1}, \overline{g})$ be a semi-Riemannian manifold,  where $\overline{g}$ is of signature $(-,+,+,+,+)$ with respect to canonical basis $(\partial x_{1},\partial x_{2},\partial x_{3},\partial x_{4},\ \partial x_{5})$, where $(x_{1},\cdots,x_{5})$ are the usual coordinates on $\overline{M}$. Let $M$ be a submanifold of $\overline{M}$ and given parametrically by the following equations
 \begin{align*}
  x_{1}=&\varphi_{1},\;\; x_{2}= \sin \varphi_{2}\sin \varphi_{3},\;\; x_{3}=\varphi_{1},\;\; x_{4}=\cos \varphi_{2}\sin \varphi_{3},\\
  &x_{5}=\cos \varphi_{3},\;\;\mbox{where}\;\; \varphi_{2}\in[0,2\pi]\;\; \mbox{and}\;\; \varphi_{3}\in(0,\pi/2).
 \end{align*}
Then we have $TM=\mathrm{span}\{E,Z_{1},Z_{2}\}$ and $l\mathrm{tr}(TM)=\mathrm{span}\{N\}$,  where 
\begin{align*}
 &E=\partial x_{1}+\partial x_{3},\;\; Z_{1}=\cos \varphi_{3}\partial x_{2}-\sin \varphi_{2}\sin \varphi_{3}\partial x_{5},\\
 &Z_{2}=\cos \varphi_{3}\partial x_{4}-\cos \varphi_{2}\sin \varphi_{3}\partial x_{5}\;\;\mbox{and}\;\; N=\frac{1}{2}(-\partial x_{1}+\partial x_{3}).
\end{align*}
Also, by straightforward calculations, we have 
\begin{align*}
 W=\sin \varphi_{2}\sin \varphi_{3} \partial x_{2}+\cos \varphi_{2}\sin \varphi_{3}\partial x_{4}+\cos \varphi_{3}\partial x_{5}.
\end{align*}
Thus, $S(TM^{\perp})=\mathrm{span}\{W\}$ and hence $M$ is a half-lightlike submanifold of $\overline{M}$. Furthermore we have $[Z_{1},Z_{2}]=\cos \varphi_{2} \sin \varphi_{3}\partial x_{2}-\sin \varphi_{2} \sin \varphi_{3}\partial x_{4}$,
which leads to $[Z_{1},Z_{2}]=\cos \varphi_{2}\tan \varphi_{3}Z_{1}-\sin \varphi_{2}\tan \varphi_{3}Z_{2}\in\Gamma(S(TM))$. Thus, $M$ is a screen integrable half-lightlike submanifold of $\overline{M}$. Finally, it is easy to see that $A_{N}$ is self-adjoint operator on $S(TM)$.
 }
\end{example}
In the next sections we shall consider screen integrable half-lightlike submanifolds of semi-Riemannian manifold $\overline{M}$ and derive special integral formulae for a foliation of $S(TM)$, whose normal vector is $\widehat{W}$ and with shape operator $\mathcal{A}_{\widehat{W}}$.

\section{Newton transformations of $\mathcal{A}_{\widehat{W}}$} \label{New1}

Let $(\overline{M}^{m+3}, \overline{g})$ be a semi-Riemannian manifold and let $(M^{n+1},g)$ be a screen integrable half-lightlike submanifold of $\overline{M}$. Then $S(TM)$ admits a foliation and let $\mathcal{F}$ be a such foliation. Then, the leaves of $\mathcal{F}$ are  co-dimension three submanifolds of $\overline{M}$, whose normal bundle is $S(TM)^{\perp}$. Let $\widehat {W}$ be unit normal vector to $\mathcal{F}$ such that the orientation of $\overline{M}$ coincides with that given by $\mathcal{F}$ and $\widehat{W}$. The Levi-Civita connection $\overline{\nabla}$ on the tangent bundle of $\overline{M}$ induces a metric connection $\nabla'$ on $\mathcal{F}$. Furthermore, $h'$ and $\mathcal{A}_{\widehat{W}}$ are the second fundamental form and shape operator of $\mathcal{F}$. Notice that $\mathcal{A}_{\widehat{W}}$ is self-adjoint on $T\mathcal{F}$ and at each point $p\in \mathcal{F}$ has $n$ real eigenvalues (or principal curvatures) $\kappa_{1}(p),\cdots,\kappa_{n}(p)$. Attached to the shape  operator $\mathcal{A}_{\widehat{W}}$ are $n$ algebraic invariants
\begin{equation*}
 S_{r}=\sigma_{r}(\kappa_{1},\cdots, \kappa_{n}),\;\; 1\le r\le n,
\end{equation*}
where $\sigma_{r}: M^{'n}\rightarrow\mathbb{R}$ are symmetric functions given by
\begin{equation}\label{N71}
 \sigma_{r}( \kappa_{1},\cdots, \kappa_{n})=\sum_{1\leq i_{1}< \cdots< i_{r}\leq n} \kappa_{i_{1}}\cdots \kappa_{i_{r}}.
\end{equation}
 Then, the  characteristic polynomial of $\mathcal{A}_{\widehat{W}}$ is given by 
\begin{equation*}
 \det(\mathcal{A}_{\widehat{W}}-t\mathbb{I})=\sum_{\alpha=0}^{n}(-1)^{\alpha}S_{r}t^{n-\alpha},
\end{equation*}
where $\mathbb{I}$ is the identity in $\Gamma(T\mathcal{F})$. The normalized $r$-th mean curvature $H_{r}$ of $M'$ is defined by
\begin{equation*}
 H_{r}=\dbinom{n}{r}^{-1}S_{r}\;\;\;\mbox{and}\;\;\;H_{0}=1. \;\;\mbox{(a constant function 1)}.
\end{equation*}
In particular, when $r=1$ then $H_{1}=\frac{1}{n} \mathrm{tr}(\mathcal{A}_{\widehat{W}})$ which is  the \textit{mean curvature} of a   $\mathcal{F}$. On the other hand, $H_{2}$ relates directly with the (intrinsic) scalar curvature of $\mathcal{F}$. Moreover, the functions $S_{r}$ ($H_{r}$ respectively) are smooth on the whole $M$ and, for any point $p\in \mathcal{F}$, $S_{r}$ coincides with the $r$-th mean curvature at $p$. In this paper, we shall use $S_{r}$ instead of $H_{r}$.

Next, we introduce the Newton transformations with respect to the operator  $\mathcal{A}_{\widehat{W}}$.  The Newton transformations $T_{r}:\Gamma(T\mathcal{F})\rightarrow \Gamma(T\mathcal{F})$ of a foliation  $\mathcal{F}$ of a screen integrable half-lightlike submanifold $M$ of an $(n+3)$-dimensional semi-Riemannian manifold $\overline{M}$ with respect to $A_{\widehat W}$ are given by 
 by the inductive formula
\begin{equation}\label{N28}
T_{0}=\mathbb{I},\quad T_{r}=(-1)^{r}S_{r}\mathbb{I}+\mathcal{A}_{\widehat{W}}\circ T_{r-1}, \;\;\; 1\leq r\leq n.
\end{equation}
By Cayley-Hamiliton theorem, we have $T_{n}=0$. Moreover, $T_{r}$ are also self-adjoint and commutes with $\mathcal{A}_{\widehat{W}}$. Furthermore, the following algebraic properties of $T_{r}$ are well-known (see \cite{krz}, \cite{woj} and references therein for details).
\begin{align}
 \mathrm{tr}(T_{r}) & =(-1)^{r}(n-r) S_{r},\label{N31}\\
 \mathrm{tr}(\mathcal{A}_{\widehat{W}}\circ T_{r}) & =(-1)^{r}(r+1)S_{r+1},\label{N32}\\
 \mathrm{tr}(\mathcal{A}_{\widehat{W}}^{2}\circ T_{r}) & =(-1)^{r+1}(-S_{1} S_{r+1} + (r + 2) S_{r+2}),\label{N33}\\
 \mathrm{tr}(T_{r}\circ\nabla'_{X} \mathcal{A}_{\widehat{W}}) & = (-1)^{r}X(S_{r+1})=(-1)^{r}\overline{g}(\nabla'S_{r+1},X),.\label{N34}
\end{align}
for all $X\in\Gamma(T\overline{M})$.
We will also need the following divergence formula for the operators $T_{r}$
\begin{equation}\label{N35}
 \mathrm{div}^{\nabla'}(T_{r})=\mathrm{tr}(\nabla' T_{r})=\sum_{\beta=1}^{n}(\nabla'_{Z_{\beta}}T_{r})Z_{\beta},
\end{equation}
where $\{Z_{1},\cdots,Z_{n}\}$ is a local orthonormal frame field  of $T\mathcal{F}$.

\section{Integration formulae for $\mathcal{F}$}\label{dist}

This section is devoted to derivation of integral formulas of foliation $\mathcal{F}$ of $S(TM)$ with a unit normal vector $\widehat{W}$ given by (\ref{P16}). By the fact that $\overline{\nabla}$ is a metric connection then  $\overline{g}(\overline{\nabla}_{\widehat{W}}\widehat{W}, \widehat{W})=0$. This implies that the vector field $\overline{\nabla}_{\widehat{W}}\widehat{W}$ is always tangent to $\mathcal{F}$. Our main goal will be to compute the divergence of the vectors $T_{r}\overline{\nabla}_{\widehat{W}}\widehat{W}$ and $T_{r}\overline{\nabla}_{\widehat{W}}\widehat{W}+(-1)^{r}S_{r+1}\widehat{W}$. The following technical lemmas are fundamentally important to this paper.
Let $\{E,Z_{i},N,W\}$, for $i=1,\cdots,n$ be a quasi-orthonormal field of frame of $T\overline{M}$, such that $S(TM)=\mathrm{span}\{Z_{i}\}$ and $\epsilon_{i}=\overline{g}(Z_{i},Z_{i})$. 
\begin{lemma}\label{lemma1}
 Let $M$ be a screen integrable half-lightlike submanifold of $\overline{M}^{n+3}$ and let $M'$ be a foliation of $S(TM)$. Let $\mathcal{A}_{\widehat{W}}$ be its shape operator, where $\widehat{W}$ is a unit normal vector to $\mathcal{F}$. Then
 \begin{align*}
  \overline{g}((\nabla'_{X}\mathcal{A}_{\widehat{W}})Y,Z)=\overline{g}(Y,(\nabla'_{X}\mathcal{A}_{\widehat{W}})Z),\;\;\overline{g}((\nabla'_{X}T_{r})Y,Z)=\overline{g}(Y,(\nabla'_{X}T_{r})Z),
 \end{align*}
for all $X,Y,Z\in\Gamma(T\mathcal{F})$.
\end{lemma}

\begin{proof}
 By simple calculations we have 
 \begin{align}\label{P20}
  \overline{g}((\nabla'_{X}\mathcal{A}_{\widehat{W}})Y,Z)=\overline{g}(\nabla'_{X}(\mathcal{A}_{\widehat{W}}Y),Z)-\overline{g}(\nabla'_{X}Y,\mathcal{A}_{\widehat{W}}Z).
 \end{align}
Using the fact that $\nabla'$ is a metric connection and the symmetry of $\mathcal{A}_{\widehat{W}}$, (\ref{P20}) gives
\begin{align}\label{P21}
 \overline{g}((\nabla'_{X}\mathcal{A}_{\widehat{W}})Y,Z)=\overline{g}(Y,\nabla'_{X}(\mathcal{A}_{\widehat{W}}Z))-\overline{g}(Y,\mathcal{A}_{\widehat{W}}(\nabla'_{X}Z)).
\end{align}
Then, from (\ref{P21}) we deduce the first relation of the lemma. A proof of the second relation follows in the same way, which completes the proof.
\end{proof}
\begin{lemma}\label{lemma2}
  Let $M$ be a screen integrable half-lightlike submanifold of $\overline{M}$ and let $\mathcal{F}$ be a co-dimension three foliation of $S(TM)$. Let $\mathcal{A}_{\widehat{W}}$ be its shape operator, where $\widehat{W}$ is a unit normal vector to $\mathcal{F}$. Denote by $\overline{R}$ the curvature tensor of $\overline{M}$. Then
  \begin{align*}
   \mathrm{div}^{\nabla'}(T_{0})&=0,\\
   \mathrm{div}^{\nabla'}(T_{r})& =\mathcal{A}_{\widehat{W}}\mathrm{div}^{\nabla'}(T_{r-1}) + \sum_{i=1}^{n}\epsilon_{i}(\overline{R}(\widehat {W},T_{r-1}Z_{i})Z_{i})',
  \end{align*}
where $(\overline{R}(\widehat {W},X)Z)'$ denotes the tangential component of $\overline{R}(\widehat {W},X)Z$ for $X,Z\in\Gamma(T\mathcal{F})$. Equivalently, for any $Y\in\Gamma(T\mathcal{F})$ then
\begin{align}\label{P22}
 \overline{g}(\mathrm{div}^{\nabla'}(T_{r}),Y)=\sum_{j=1}^{r}\sum_{i=1}^{n} \epsilon_{i}\overline{g}(\overline{R}(T_{r-1}Z_{i},\widehat{W})(-\mathcal{A}_{\widehat{W}})^{j-1}Y,Z_{i}).
\end{align}
\end{lemma}
\begin{proof}
The first equation of the lemma is obvious since $T_{0}=\mathbb{I}$. We turn to the second relation. By direct calculations using the recurrence relation (\ref{N28}) we derive
\begin{align}\label{P23}
 \mathrm{div}^{\nabla'}(T_{r})&=(-1)^{r}\mathrm{div}^{\nabla'}(S_{r}\mathbb{I})+\mathrm{div}^{\nabla'}(\mathcal{A}_{\widehat{W}}\circ T_{r-1})\nonumber\\
  &=(-1)^{r}\nabla'S_{r}+\mathcal{A}_{\widehat{W}}\mathrm{div}^{\nabla'}(T_{r-1})+\sum_{i=1}^{n}\epsilon_{i}(\nabla'_{Z_{i}}\mathcal{A}_{\widehat{W}})T_{r-1}Z_{i}. 
\end{align}
Using Codazzi equation 
\begin{equation*}
 \overline{g}(\overline{R}(X,Y)Z,\widehat{W})=\overline{g}((\nabla'_{Y}\mathcal{A}_{\widehat{W}})X,Z)-\overline{g}((\nabla'_{X}\mathcal{A}_{\widehat{W}})Y,Z),
\end{equation*}
for any $X,Y,Z\in\Gamma(T\mathcal{F})$ and Lemma \ref{lemma1}, we have 
\begin{align}\label{P24}
 \overline{g}((\nabla'_{Z_{i}}\mathcal{A}_{\widehat{W}})Y,&T_{r-1}Z_{i})= \overline{g}((\nabla'_{Y}\mathcal{A}_{\widehat{W}})Z_{i},T_{r-1}Z_{i})+\overline{g}(\overline{R}(Y,Z_{i})T_{r-1}Z_{i},\widehat{W})\nonumber\\
 &=\overline{g}(T_{r-1}(\nabla'_{Y}\mathcal{A}_{\widehat{W}})Z_{i},Z_{i})+\overline{g}(\overline{R}(\widehat{W},T_{r-1}Z_{i})Z_{i},Y),
\end{align}
for any $Y\in\Gamma(T\mathcal{F})$. Then applying (\ref{P23}) and (\ref{P24}) we get
\begin{align}\label{P25}
\overline{g}(\mathrm{div}^{\nabla'}(T_{r}), Y)&=(-1)^{r}\overline{g}(\nabla'S_{r},Y)+\mathrm{tr}(T_{r-1}(\nabla'_{Y}\mathcal{A}_{\widehat{W}}))\nonumber\\
&+\overline{g}(\mathrm{div}^{\nabla'}(T_{r-1}),Y) + \overline{g}(Y, \sum_{i=1}^{n}\epsilon_{i}\overline{R}(\widehat{W},T_{r-1}Z_{i})Z_{i}).
\end{align}
Then, applying (\ref{P25}) and (\ref{N34}) we get the second equation of the lemma. Finally, (\ref{P22}) follows immediately by an induction argument.
\end{proof}
Notice that when the ambient manifold is a  space form of constant sectional
curvature, then $(\overline{R}(\widehat{W}, X )Y)'= 0$ for each $X,Y\in\Gamma(T\mathcal{F})$. Hence, from Lemma (\ref{lemma2}) we have $\mathrm{div}^{\nabla'}(T_{r})=0$.

\begin{lemma}\label{lemma3}
 Let $M$ be a screen integrable half-lightlike submanifold of $\overline{M}$ and let $\mathcal{F}$ be a co-dimension three foliation of $S(TM)$. Let $\mathcal{A}_{\widehat{W}}$ be its shape operator, where $\widehat{W}$ is a unit normal vector to $\mathcal{F}$. Let $\{Z_{i}\}$ be a local field such $(\nabla'_{X}Z_{i})p=0$, for $i=1,\cdots,n$ and any vector field $X\in\Gamma(T\overline{M})$. Then at $p\in \mathcal{F}$ we have 
 \begin{align*}
  g(\nabla'_{Z_{i}}\overline{\nabla}_{\widehat{W}}\widehat{W},Z_{j})&=g(\mathcal{A}_{\widehat{W}}^{2}Z_{i},Z_{j})-\overline{g}(\overline{R}(Z_{i},\widehat{W})Z_{j}, \widehat{W})\\
  &-\overline{g}((\nabla'_{\widehat{W}}\mathcal{A}_{\widehat{W}})Z_{i},Z_{j})+g(\overline{\nabla}_{\widehat{W}}\widehat{W},Z_{i})g(Z_{j},\overline{\nabla}_{\widehat{W}}\widehat{W}).
 \end{align*}
\end{lemma}
\begin{proof}
 Applying $\overline{\nabla}_{Z_{i}}$ to $g(\overline{\nabla}_{\widehat{W}}\widehat{W},Z_{j})$ and $\overline{g}(\widehat{W},\overline{\nabla}_{\widehat{W}}Z_{j})$ in turn and then using the two resulting equations, we have
 \begin{align}\label{P26}
  -\overline{g}(\overline{\nabla}_{\widehat{W}}\widehat{W},\overline{\nabla}_{Z_{i}}Z_{j})&=g(\overline{\nabla}_{Z_{i}}\overline{\nabla}_{\widehat{W}}\widehat{W},Z_{j})+ \overline{g}(\overline{\nabla}_{Z_{i}}\widehat{W},\overline{\nabla}_{\widehat{W}}Z_{j})\nonumber\\
  &+\overline{g}(\widehat{W},\overline{\nabla}_{Z_{i}}\overline{\nabla}_{\widehat{W}}Z_{j}).
 \end{align}
Furthermore, by direct calculations using $(\nabla'_{X}Z_{i})p=0$ we have 
\begin{align*}
 \overline{g}((\nabla'_{\widehat{W}}\mathcal{A}_{\widehat{W}})Z_{i},Z_{j})= \overline{g}(\overline{\nabla}_{\widehat{W}}\widehat{W},\overline{Z_{i}}Z_{j})+\overline{g}(\widehat{W},\overline{\nabla}_{\widehat{W}}\overline{Z_{i}}Z_{j}),
\end{align*}
and hence 
\begin{align}\label{P27}
 g(\mathcal{A}_{\widehat{W}}^{2}Z_{i},Z_{j})-\overline{g}(\overline{R}&(Z_{i},\widehat{W})Z_{j}, \widehat{W})-\overline{g}((\nabla'_{\widehat{W}}\mathcal{A}_{\widehat{W}})Z_{i},Z_{j})\nonumber\\
&= g(\mathcal{A}_{\widehat{W}}^{2}Z_{i},Z_{j})-\overline{g}(\overline{R}(Z_{i},\widehat{W})Z_{j}, \widehat{W})\nonumber\\
&-\overline{g}(\overline{\nabla}_{\widehat{W}}\widehat{W},\overline{Z_{i}}Z_{j})-\overline{g}(\widehat{W},\overline{\nabla}_{\widehat{W}}\overline{Z_{i}}Z_{j})\nonumber\\
&=g(\mathcal{A}_{\widehat{W}}^{2}Z_{i},Z_{j})-\overline{g}(\overline{\nabla}_{Z_{i}}Z_{j},\overline{\nabla}_{\widehat{W}}\widehat{W})\nonumber\\
&-\overline{g}(\overline{\nabla}_{Z_{i}}\overline{\nabla}_{\widehat{W}}Z_{j},\widehat{W})+\overline{g}(\overline{\nabla}_{[Z_{i},\widehat{W}]}Z_{j},\widehat{W}).
\end{align}
Now, applying (\ref{P26}), the condition at $p$ and the following relations 
\begin{align*}
 &\overline{\nabla}_{Z_{i}}\widehat{W}=\sum_{k=1}^{n}\epsilon_{k}\overline{g}( \overline{\nabla}_{Z_{i}}\widehat{W},Z_{k})Z_{k},\;\;\;\overline{\nabla}_{\widehat{W}}Z_{j}=\overline{g}(\overline{\nabla}_{\widehat{W}}Z_{j}, \widehat{W})\widehat{W},
\end{align*}
and $
  g(\mathcal{A}_{\widehat{W}}^{2}Z_{i},Z_{j})=-\sum_{k=1}^{n}\epsilon_{k}\overline{g}( \overline{\nabla}_{Z_{i}}\widehat{W},Z_{k})\overline{g}(\overline{\nabla}_{Z_{k}}Z_{j},\widehat{W})
$
to the last line of (\ref{P27}) and the fact that $S(TM)$ is integrable we get
\begin{align*}
 &g(\mathcal{A}_{\widehat{W}}^{2}Z_{i},Z_{j})-\overline{g}(\overline{R}(Z_{i},\widehat{W})Z_{j}, \widehat{W})-\overline{g}((\nabla'_{\widehat{W}}\mathcal{A}_{\widehat{W}})Z_{i},Z_{j})\\
 &=g(\nabla'_{Z_{i}}\overline{\nabla}_{\widehat{W}}\widehat{W},Z_{j})-g(\overline{\nabla}_{\widehat{W}}\widehat{W},Z_{i})g(Z_{j},\overline{\nabla}_{\widehat{W}}\widehat{W}),
\end{align*}
from which the lemma follows by rearrangement.
\end{proof}
Notice that, using parallel transport, we can always construct a frame field from the above lemma.
\begin{proposition}\label{proposition1}
 Let $M$ be a screen integrable half-lightlike submanifold of an indefinite almost contact manifold $\overline{M}$ and let $\mathcal{F}$ be a foliation of $S(TM)$. Then
 \begin{align*}
  &\mathrm{div}^{\nabla'}(T_{r}\overline{\nabla}_{\widehat{W}}\widehat{W})=\overline{g}( \mathrm{div}^{\nabla'}(T_{r}),\overline{\nabla}_{\widehat{W}}\widehat{W})+(-1)^{r+1}\widehat{W}(S_{r+1})\\
  &+(-1)^{r+1}(-S_{1}S_{r+1}+(r+2)S_{r+2})-\sum_{i=1}^{n}\epsilon_{i}\overline{g}(\overline{R}(Z_{i},\widehat{W})T_{r}Z_{i},\widehat{W})\\
  &+\overline{g}(\overline{\nabla}_{\widehat{W}}\widehat{W},T_{r}\overline{\nabla}_{\widehat{W}}\widehat{W}),
 \end{align*}
where $\{ Z_{i}\}$ is a field of frame tangent to the leaves of $\mathcal{F}$.
\end{proposition}

\begin{proof}
 From (\ref{N35}),we deduce that 
 \begin{align}\label{P30}
  \mathrm{div}^{\nabla'}(T_{r}Z)=\overline{g}( \mathrm{div}^{\nabla'}(T_{r}),Z)+\sum_{i=1}^{n}\epsilon_{i}\overline{g}(\nabla'_{Z_{i}}Z,T_{r}Z_{i} ),
 \end{align}
for all $Z\in\Gamma(T\mathcal{F})$. Then using (\ref{P30}), Lemmas \ref{lemma2} and \ref{lemma3}, we obtain the desired result. Hence the proof.
\end{proof}
From Proposition \ref{proposition1} we have 
\begin{theorem}\label{theorem1}
  Let $M$ be a screen integrable half-lightlike submanifold of an indefinite almost contact manifold $\overline{M}$ and let $\mathcal{F}$ be a co-dimension three foliation of $S(TM)$. Then 
  \begin{align*}
    &\mathrm{div}^{\overline{\nabla}}(T_{r}\overline{\nabla}_{\widehat{W}}\widehat{W})=\overline{g}( \mathrm{div}^{\nabla'}(T_{r}),\overline{\nabla}_{\widehat{W}}\widehat{W})+(-1)^{r+1}\widehat{W}(S_{r+1})\\
  &+(-1)^{r+1}(-S_{1}S_{r+1}+(r+2)S_{r+2})-\sum_{i=1}^{n}\epsilon_{i}\overline{g}(\overline{R}(Z_{i},\widehat{W})T_{r}Z_{i},\widehat{W}).
 \end{align*}
\end{theorem}
\begin{proof}
 A proof follows easily form Proposition \ref{proposition1} by recognizing the fact that 
 \begin{align*}
  \mathrm{div}^{\overline{\nabla}}(T_{r}\overline{\nabla}_{\widehat{W}}\widehat{W})=\mathrm{div}^{\nabla'}(T_{r}\overline{\nabla}_{\widehat{W}}\widehat{W})-\overline{g}(\overline{\nabla}_{\widehat{W}}\widehat{W},T_{r}\overline{\nabla}_{\widehat{W}}\widehat{W}),
 \end{align*}
which completes the proof.
\end{proof}
\begin{theorem}\label{theorem2}
 Let $M$ be a screen integrable half-lightlike submanifold of  $\overline{M}$ and let $\mathcal{F}$ be a co-dimension three foliation of $S(TM)$. Then, 
 \begin{align*}
  &\mathrm{div}^{\overline{\nabla}}(T_{r}\overline{\nabla}_{\widehat{W}}\widehat{W}+(-1)^{r}S_{r+1}\widehat{W})=\overline{g}( \mathrm{div}^{\nabla'}(T_{r}),\overline{\nabla}_{\widehat{W}}\widehat{W})\\
  &+(-1)^{r+1}(r+2)S_{r+2}-\sum_{i=1}^{n}\epsilon_{i}\overline{g}(\overline{R}(Z_{i},\widehat{W})T_{r}Z_{i},\widehat{W}).
 \end{align*}

\end{theorem}
\begin{proof}
 By straightforward calculations we have 
 \begin{align*}
  S_{1}=\mathrm{tr}(\mathcal{A}_{\widehat{W}})=-\sum_{i=1}^{n}\epsilon_{i}\overline{g}(\overline{\nabla}_{Z_{i}}\widehat{W},Z_{i})=-\sum_{i=1}^{n+1}\epsilon_{i}\overline{g}(\overline{\nabla}_{Z_{i}}\widehat{W},Z_{i})=-\mathrm{div}^{\overline{\nabla}}(\widehat{W}),
 \end{align*}
where $Z_{n+1}=\widehat{W}$. From this equation we deduce
\begin{align}\label{P31}
 \mathrm{div}^{\overline{\nabla}}(S_{r+1}\widehat{W})=-S_{1}S_{r+1}+\widehat{W}(S_{r+1}).
\end{align}
Then from (\ref{P31}) and Theorem \ref{theorem1} we get our assertion, hence the proof.
\end{proof}
Next, we let $dV$ denote the volume form $\overline{M}$. Then from Theorem \ref{theorem2} we
\begin{corollary}\label{corollary1}
 Let $M$ be a screen integrable half-lightlike submanifold of a compact semi-Riemannian manifold $\overline{M}$ and let $\mathcal{F}$ be a co-dimension three foliation of $S(TM)$. Then
 \begin{align*}
  &\int_{\overline{M}}\overline{g}( \mathrm{div}^{\nabla'}(T_{r}),\overline{\nabla}_{\widehat{W}}\widehat{W})dV\\
  &=\int_{\overline{M}}((-1)^{r}(r+2)S_{r+2}+\sum_{i=1}^{n}\epsilon_{i}\overline{g}(\overline{R}(Z_{i},\widehat{W})T_{r}Z_{i},\widehat{W})dV
 \end{align*}

\end{corollary}
Setting $r=0$ in the above corollary we get 

\begin{corollary}\label{corollary3}
 Let $M$ be a screen integrable half-lightlike submanifold of a compact semi-Riemannian manifold $\overline{M}$ and let $\mathcal{F}$ be a co-dimension three foliation of $S(TM)$ with mean curvatures $S_{r}$. Then for $r=0$ we have 
 \begin{align*}
 \int_{\overline{M}}2S_{2}dV=\int_{\overline{M}}\overline{Ric}(\widehat{W},\widehat{W})dV,
\end{align*}
where $\displaystyle\overline{Ric}(\widehat{W},\widehat{W})=\sum_{i=1}^{n}\epsilon_{i}\overline{g}(\overline{R}(Z_{i},\widehat{W})\widehat{W},Z_{i})$.
\end{corollary}
Notice that the equation in Corollary \ref{corollary3} is the lightlike analogue of (3.5) in \cite{krz} for co-dimension one foliations on Riemannian manifolds.

Next,  we will discuss some consequences of the integral formulas developed so far. 

 A semi-Riemannian manifold $\overline{M}$ of constant sectional curvature $c$ is called a \textit{semi-Riemannian space form} \cite{db, ds2} and is denoted by $\overline{M}(c)$. Then, the curvature tensor $\overline{R}$ of $\overline{M}(c)$ is given by 
\begin{equation}\label{N24}
 \overline{R}(\overline{X},\overline{Y})\overline{Z}=c\{\overline{g}(\overline{Y},\overline{Z})\overline{X}-\overline{g}(\overline{X},\overline{Z})\overline{Y}\},\quad\forall\, \overline{X},\overline{Y},\overline{Z}\in\Gamma(T\overline{M}).
\end{equation} 

\begin{theorem}\label{theorem3}
 Let $M$ be a screen integrable half-lightlike submanifold of a compact semi-Riemannian space form $\overline{M}(c)$ of constant sectional curvature $c$. Let $\mathcal{F}$ be a co-dimension three foliation of its screen distribution $S(TM)$. If $V$ is the total volume of $\overline{M}$, then
 \begin{equation}
  \int_{\overline{M}}S_{r}dV=\begin{dcases}
0, & r=2k+1, \\
c^{\frac{r}{2}}\dbinom{\frac{n}{2}}{\frac{r}{2}}V, & r=2k,
\end{dcases}
 \end{equation}
for positive integers $k$.
\end{theorem}
\begin{proof}
 By setting $\overline{X}=Z_{i}$, $\overline{Y}=\widehat{W}$ and $Z=T_{r}Z_{i}$ in (\ref{N24}) we deduce that $ \overline{R}(Z_{i},\widehat{W})T_{r}Z_{i}=-cg(Z_{i},T_{r}Z_{i})\widehat{W}$. Then substituting this equation in Corollary \ref{corollary1} we obtain
 \begin{align*}
  \int_{\overline{M}}\overline{g}( \mathrm{div}^{\nabla'}(T_{r}),\overline{\nabla}_{\widehat{W}}\widehat{W})dV=\int_{\overline{M}}((-1)^{r}(r+2)S_{r+2}-c\mathrm{tr}(T_{r}))dV.
 \end{align*}
Since $\overline{M}$ is of constant sectional curvature $c$, then Lemma \ref{lemma2} implies that $T_{r}=0$ for any $r$ and hence the above equation simplifies to 
\begin{align}\label{P35}
 (r+2)\int_{\overline{M}}S_{r+2}dV=c(n-r)\int_{\overline{M}}S_{r}dV.
\end{align}
Since $S_{1}=-\mathrm{div}^{\overline{\nabla}}(\widehat{W})$ and that $\overline{M}$ is compact, then $\int_{\overline{M}}S_{1}dV=0$. Using this fact together with (\ref{P35}), mathematical induction gives $\int_{\overline{M}}S_{r}dV=0$ for all $r=2k+1$ (i.e., $r$ odd). For $r$ even we will consider $r=2m$ and $n=2l$ (i.e., both $M$ and $\overline{M}$ are odd dimensional). With these conditions, (\ref{P35}) reduces to 
\begin{align}\label{P36}
 \int_{\overline{M}}S_{2m+2}dV=c\frac{l-m}{1+m}\int_{\overline{M}}S_{2m}dV.
\end{align}
Now setting $m=0,1,\cdots$ and $S_{0}=1$ in (\ref{P36}) we obtain 
\begin{align*}
 \int_{\overline{M}}S_{2}dV=cl V, \;\;\;\; \int_{\overline{M}}S_{4}dV=c^{2}\frac{(l-1)l}{2}V, 
\end{align*}
and more generally
\begin{align}\label{P37}
 \int_{\overline{M}}S_{2k}dV=c^{k}\frac{(l-k+1)(l-k+2)(l-k+3)\cdots l}{k!}V.
\end{align}
Hence, our assertion follows from \ref{P37}, which completes the proof.
\end{proof}
Next, when $\overline{M}$ is Einstein i.e., $\overline{Ric}=\mu \overline{g}$ we have the following.
\begin{theorem}\label{theorem5}
 Let $M$ be a screen integrable half-lightlike submanifold of an Einstein compact semi-Riemannian manifold $\overline{M}$. Let $\mathcal{F}$ be a co-dimension three foliation of its screen distribution $S(TM)$ with totally umbilical leaves. If $V$ is the total volume of $\overline{M}$, then
 \begin{equation}
  \int_{\overline{M}}S_{r}dV=\begin{dcases}
0, & r=2k+1, \\
\left(\frac{\mu}{n}\right)^{\frac{n}{2}}\dbinom{\frac{n}{2}}{\frac{r}{2}}V, & r=2k,
\end{dcases}
 \end{equation}
for positive integers $k$.
\end{theorem}
\begin{proof}
 Suppose that $\mathcal{A}_{\widehat{W}}=\frac{1}{n}S_{r}\mathbb{I}$. Then by direct calculations using the formula for $T_{r}$ we derive $T_{r}=(-1)^{r+1}\frac{(n-r)}{n}S_{r}\mathbb{I}$. Then, under the assumptions of the theorem we obtain $\overline{Ric}(\widehat {W}, \overline{\nabla}_{\widehat{W}}\widehat{W})=0$ and $\overline{Ric}(\widehat {W},\widehat{W})=\mu$ and hence, Corollary \ref{corollary1} reduces to 
 \begin{align}\label{P50}
  n(r+2)\int_{\overline{M}}S_{r+2}dV=\lambda(n-r)\int_{\overline{M}}S_{r}dV.
 \end{align}
Notice that (\ref{P50}) is similar to (\ref{P35}) and hence following similar steps as in the previous theorem we get $\int_{\overline{M}}S_{r}dV=0$ for $r$ odd and for $r$ even we get 
\begin{align*}
 \int_{\overline{M}}S_{2k}dV=\left(\frac{\mu}{n}\right)^{k}\frac{(l-k+1)(l-k+2)(l-k+3)\cdots l}{k!}V,
\end{align*}
which complete the proof.
\end{proof}

\section{Screen umbilical half-lightlike submanifolds}\label{screen}

In this section we consider totally umbilical half-lightlike submanifolds of semi-Riemannian manifold, with a totally umbilical screen distribution and thus, give a generalized version of Theorem 4.3.7 of \cite{ds2} and its Corollaries, via Newton transformations of the operator $A_{N}$.

A screen distribution $S(TM)$ of a half-lightlike submanifold $M$ of a semi-Riemannian manifold $\overline{M}$ is said to be totally umbilical \cite{ds2} if on any coordinate neighborhood $\mathcal{U}$ there exist a function $K$ such that 
\begin{align}\label{P64}
 C(X,PY)=Kg(X,PY),\;\;\;\forall\, X,Y\in\Gamma(TM).
\end{align}
In case $K=0$, we say that $S(TM)$ is totally geodesic. Furthermore, if $S(TM)$ is totally umbilical then by straightforward calculations we have 
\begin{align}\label{P65}
 A_{N}X=PX,\;\;\; C(E,PX)=0,\;\;\; \forall\, X\in\Gamma(TM).
\end{align}
Let $\{E,Z_{i}\}$, for $i=1,\cdots,n$, be a quasi-orthonormal frame field of $TM$ which diagonalizes $A_{N}$. Let $l_{0},l_{1},\cdots,l_{n}$ be the respective eigenvalues (or principal curvatures). Then as before, the $r$-th mean curvature $S_{r}$ is given by 
\begin{align*}
 S_{r}=\sigma_{r}(l_{0},\cdots,l_{n}) \;\;\mbox{and}\;\;S_{0}=1.
\end{align*}
The characteristic polynomial of $A_{N}$ is given by 
\begin{equation*}
 \det(A_{N}-t\mathbb{I})=\sum_{\alpha=0}^{n}(-1)^{\alpha}S_{r}t^{n-\alpha},
\end{equation*}
where $\mathbb{I}$ is the identity in $\Gamma(TM)$. The normalized $r$-th mean curvature $H_{r}$ of $M$ is defined by
$\dbinom{n}{r}H_{r}=S_{r}\;\;\;\mbox{and}\;\;\;H_{0}=1$.  The Newton transformations $T_{r}:\Gamma(TM)\rightarrow \Gamma(TM)$ of  $A_{N}$ are given by the inductive formula
\begin{equation}\label{K28}
T_{0}=\mathbb{I},\quad T_{r}=(-1)^{r}S_{r}\mathbb{I}+A_{N}\circ T_{r-1}, \;\;\; 1\leq r\leq n.
\end{equation}
By Cayley-Hamiliton theorem, we have $T_{n+1}=0$. Also, $T_{r}$ satisfies the following properties.
\begin{align}
 \mathrm{tr}(T_{r}) & =(-1)^{r}(n+1-r) S_{r},\label{K31}\\
 \mathrm{tr}(A_{N}\circ T_{r}) & =(-1)^{r}(r+1)S_{r+1},\label{K32}\\
 \mathrm{tr}(A_{N}^{2}\circ T_{r}) & =(-1)^{r+1}(-S_{1} S_{r+1} + (r + 2) S_{r+2}),\label{K33}\\
 \mathrm{tr}(T_{r}\circ\nabla_{X} A_{N})& = (-1)^{r}X(S_{r+1}).\label{K34}
\end{align}
for all $X\in\Gamma(TM)$.
\begin{proposition}\label{pp1}
 Let $(M,g)$ be a totally umbilical half-lightlike submanifold of a semi-Riemannian manifold $\overline{M}$ of constant sectional curvature $c$. Then
 \begin{align*}
  & g(\mathrm{div}^{\nabla}(T_{r}),X)  =(-1)^{r-1}\lambda(X)E(S_{r})-\tau(X)\mathrm{tr}(A_{N}\circ T_{r-1})\\
  & -c\lambda(X)\mathrm{tr}(T_{r-1}) +g(\mathrm{div}^{\nabla}(T_{r-1}),A_{N}X)+g((\nabla_{E}A_{N})T_{r-1}E,X)\\
  &+\sum_{i=1}^{n}\epsilon_{i}\{-\lambda(X)B(Z_{i},A_{N}(T_{r-1}Z_{i}))\\
&+\varepsilon \tau(Z_{i})C(X,T_{r-1}Z_{i})\{\rho(X)D(Z_{i},T_{r-1}Z_{i})-\rho(Z_{i})D(X,T_{r-1}Z_{i})\}\},
 \end{align*}
for any $X\in\Gamma(TM)$.
\end{proposition}
\begin{proof}
 From the recurrence relation (\ref{K28}), we derive 
\begin{align}\label{P66}
 &g(\mathrm{div}^{\nabla}(T_{r}),X)=(-1)^{r}PX(S_{r})+g((\nabla_{E}A_{N})T_{r-1}E,X)\nonumber\\
 &+g(\mathrm{div}^{\nabla}(T_{r-1}),A_{N}X)+\sum_{i=1}^{n}\epsilon_{i}g((\nabla_{Z_{i}}A_{N})T_{r-1}Z_{i},X).
\end{align}
for any $X\in\Gamma(TM)$.
But  
  \begin{align}\label{P67}
  g((\nabla_{Z_{i}}A_{N})T_{r-1}Z_{i},X)  &= g(T_{r-1}Z_{i},(\nabla_{Z_{i}}A_{N})X) + g(\nabla_{Z_{i}}A_{N}(T_{r-1}Z_{i}),X)\nonumber\\
                                        & -g(\nabla_{Z_{i}}(A_{N}X),T_{r-1}Z_{i}) + g(A_{N}(\nabla_{Z_{i}}X),T_{r-1}Z_{i})\nonumber\\
                                        &-g(A_{N}(\nabla_{Z_{i}}T_{r-1}Z_{i}),X),
 \end{align}
for all $X\in\Gamma(TM)$.
\end{proof}
Then applying (\ref{P6}) to (\ref{P67}) while considering the fact that $A_{N}$ is screen-valued, we get 
\begin{align}\label{P68}
 g((\nabla_{Z_{i}}A_{N})T_{r-1}Z_{i},X)= g(T_{r-1}Z_{i},(\nabla_{Z_{i}}A_{N})X)-\lambda(X)B(Z_{i},A_{N}(T_{r-1}Z_{i}))
\end{align}
Furthermore, using (\ref{P60}) and (\ref{N24}), the first term on the right hand side of (\ref{P68}) reduces to 
\begin{align}\label{P69}
 g(T_{r-1}Z_{i},&(\nabla_{Z_{i}}A_{N})X)=-c\lambda(X)g(Z_{i},T_{r-1}Z_{i})+g((\nabla_{X}A_{N})Z_{i},T_{r-1}Z_{i})\nonumber\\
 &-\tau(X)C(Z_{i},T_{r-1}Z_{i})+\varepsilon\tau(Z_{i})C(X,T_{r-1}Z_{i})\{\rho(X)D(Z_{i},T_{r-1}Z_{i})\nonumber\\
 &-\rho(X)D(X,T_{r-1}Z_{i})\},
\end{align}
for any $X\in\Gamma(TM)$. Finally, replacing (\ref{P69}) in (\ref{P68}) and then put the resulting equation in (\ref{P66}) we get the desired result.

Next, from Proposition \ref{pp1} we have the following.
\begin{theorem}\label{thmy}
 Let $(M, g)$ be a half-lightlike submanifold of a semi-Riemannian manifold $\overline{M}(c)$ of constant curvature $c$, with a proper totally umbilical screen distribution $S(TM)$. If $M$ is also totally umbilical, then the $r$-th mean curvature $S_{r}$, for $r=0,1,\cdots,n$, with respect to $A_{N}$ are solution of the following differential equation
 \begin{align*}
  E(S_{r+1})-\tau(E)(r+1)S_{r+1}-c(-1)^{r}(n+1-r)S_{r}=\mathcal{H}_{1}(r+1)S_{r+1}.
 \end{align*}
\end{theorem}
\begin{proof}
 Replacing $X$ with $E$ in the Proposition \ref{pp1} and then using  (\ref{P61}) and (\ref{P65}) we obtain, for all $r=0,1,\cdots,n$,
 \begin{align*} 
 E(S_{r+1})-(-1)^{r}\tau(E)\mathrm{tr}(A_{N}\circ T_{r})-c(-1)^{r}\mathrm{tr}(T_{r})=(-1)^{r}\mathcal{H}_{1}\mathrm{tr}(A_{N}\circ T_{r}),
 \end{align*}
from which the result follows by applying (\ref{K31}) and (\ref{K32}).
\end{proof}
\begin{corollary}\label{P71}
 Under the hypothesis of Theorem \ref{thmy}, the induced connection $\nabla$ on $M$ is a metric connection, if and only if, the $r$-th mean curvature $S_{r}$ with respect to $A_{N}$ are solution of the following equation
  \begin{align*}
  E(S_{r+1})-\tau(E)(r+1)S_{r+1}-c(-1)^{r}(n+1-r)S_{r}=0.
 \end{align*}
\end{corollary}
Also the following holds.
\begin{corollary}\label{P72}
 Under the hypothesis of Theorem \ref{thmy}, $\overline{M}(c)$ is a semi-Euclidean space, if and only if, the $r$-th mean curvature $S_{r}$ with respect to $A_{N}$ are solution of the following equation
  \begin{align*}
  E(S_{r+1})-\tau(E)(r+1)S_{r+1}=\mathcal{H}_{1}(r+1)S_{r+1}.
 \end{align*}
\end{corollary}
Notice that Theorem \ref{thmy} and Corollary \ref{P71} are generalizations of Theorem 4.3.7 and Corollary 4.3.8, respectively, given in \cite{ds2}.


\end{document}